\documentclass[11pt]{article}
\usepackage{amsfonts, amssymb, amsmath}
\usepackage[none]{hyphenat}
\usepackage{fancyhdr}
\usepackage{graphicx}
\usepackage{float}
\usepackage{lastpage}
\usepackage[latin1]{inputenc}
\usepackage{enumitem}
\usepackage{mathtools}
\usepackage{titling}

\usepackage{lipsum}
\usepackage{cite}
\usepackage{hyperref}
\usepackage{listings}
\usepackage{amsthm}
\newtheorem{theorem}{Theorem}[section]
\newtheorem{corollary}{Corollary}[section]
\newtheorem{lemma}[theorem]{Lemma}

\newtheorem{theoremd}{Theorem B}

\newtheorem{theoremB}{Theorem B2}

\newcommand{\RR}{\mathbb{R}}
\newcommand{\QQ}{\mathbb{Q}}
\newcommand{\ZZ}{\mathbb{Z}}
\newcommand{\NN}{\mathbb{N}}
\newcommand{\FF}{\mathbb{F}}
\newcommand{\vu}{\mathbf{u}}

\newcommand{\K}{\operatornamewithlimits{K}}
\newcommand{\bigk}{\operatornamewithlimits{K}}

\title{Badly approximable infinite products of quadratic polynomials}
\author{Dmitry Badziahin\; and\; Cameron Eggins}
\date{}

\begin{document}

\maketitle
\thispagestyle{empty}

\begin{abstract}
We provide a number of conditions on the rational numbers $u$ and
$v$ which ensure that the Laurent series $g_{u,v}(x):=
\prod_{t=0}^\infty (1+ux^{-3^t} + vx^{-2\cdot 3^t})$ is badly
approximable.
\end{abstract}

\section{Introduction}

Given a real number $\xi$, the irrationality exponent of $\xi$ is
defined as follows
$$
\mu(\xi):=\sup\left\{\mu\in\RR\;:\;\left| \xi - \frac{p}{q}\right| <
q^{-\mu}\;\mbox{has i.m. solutions }\; p/q\in\QQ\right\}.
$$
This is one of the most important approximational properties of a
number, which indicates, how well it can be approximated by
rationals in terms of their denominators.

In recent years there was a lot of interest in understanding the
irrationality exponents of Mahler numbers. By Mahler numbers we understand
the values of Mahler functions at integer points. The Mahler functions are in
turn analytical functions $f\in\QQ((z^{-1}))$ which for any $z$ inside their
disc of convergence satisfy the equation of the form
$$
\sum_{i=0}^n P_i(z)f(z^{d^i})=Q(z)
$$
where $n\ge 1, d\ge 2$ are integer, $P_0,\ldots, P_n, Q\in \QQ[z]$ and
$P_0P_n\neq 0$.

One of the first results in this direction was achieved in 2011 by
Bugeaud~\cite{bugeaud_2011}. He showed that the irrationality exponent of the
Thue-Morse numbers equals two. These are the numbers of the form
$f_{tm}(b):=\sum_{n=0}^\infty \frac{t_n}{b^n}$ where $b$ is integer and $t_n$
is the famous Thue-Morse sequence in $\{0,1\}$ which is recurrently defined
as follows: $t_0:=0, t_{2n}:=t_n$ and $t_{2n+1} = 1-t_{2n}$. One can easily
verify that the function $z^{-1}(1 - 2f_{tm}(z))$ satisfies the Mahler
equation
$$
f_{tm}(z) = (z-1)f_{tm}(z^2).
$$
For more results of this type, see~\cite{adamczewski_rivoal_2009},
\cite{coons_2013}, \cite{gww_2014}, \cite{badziahin_2017}.

In~\cite{bugeaud_han_wen_yao_2015}, the authors provide a non-trivial upper
bound for the irrationality exponent of $f(b)$ where the Mahler functions
satisfy
$$
Q(z) = P_0(z) f(z) + P_1(z)f(z^d).
$$
Their bound is quite general and in many cases it is sharp. But that
result is often hard to apply because it requires the knowledge
about the distribution of non-zero Hankel determinants of $f(z)$,
which are not easy to compute. Later Badziahin~\cite{badziahin_2019}
provided the precise formula for $\mu(f(b))$ for a slightly narrower
set of Mahler functions:

\begin{theoremd}
Let $f(z)\in \QQ((z^{-1}))\setminus \QQ(z)$ be a solution of the
functional equation
$$
f(z) = \frac{A(z)}{B(z)} f(z^d),\qquad A,B\in\QQ[z],\; B\neq
0,\quad d\in\ZZ,\; d\ge 2.
$$
Let $b\in \ZZ, |b|\ge 2$ be inside the disc of convergence of $f$
such that $A(b^{d^m})B(b^{d^m}) \neq 0$ for all $m\in\ZZ_{\ge 0}$.
Then
\begin{equation}\label{th1_eq}
\mu(f(b)) = 1 + \limsup_{k\to\infty}\frac{d_{k+1}}{d_k}.
\end{equation}
\end{theoremd}
Here, $d_k$ is the degree of the denominator of the $k$'th convergent of
$f(z)$. We discuss these notions in Section~\ref{sec_hankel}.

The upshot is that the irrationality measure of $f(b)$, given
by~\eqref{th1_eq} is completely determined by the sequence $d_k$.
However determining this sequence for a precise Mahler function
$f(z)$ may be problematic. In 2017, the first
author~\cite{badziahin_2017} verified that $d_k= k$ for all
functions $g_u(z)$ which satisfy $g_u(z) = (z+u)g_u(z^2)$, $u\in\QQ$
and $u\neq 0,1$. Equivalently, such functions can be written as the
infinite products
$$
g_u(z) = z^{-1}\prod_{t=0}^\infty (1+uz^{-2^t}).
$$
Notice that $g_0(z)$ and $g_1(z)$ are rational functions. Therefore
we now have a complete understanding of irrationality exponents of
$g_u(b)$.

The next natural case to investigate are the following infinite products:
$$
g_{u,v}(z) = z^{-1} \prod_{t=0}^\infty (1+uz^{-3^t} + vz^{-2\cdot
3^t}),\quad u,v\in\QQ.
$$
In~\cite{badziahin_2017}, Badziahin started the investigation of
sequences $d_k$ for various pairs $(u,v)$ of integer numbers. It was
shown that $\limsup_{k\to\infty} d_{k+1}/d_k > 1$ for:
\begin{enumerate}
\item $(u,v) = (\pm u, u^2)$, $u\in\QQ$;
\item $(u,v) = (\pm s^3, -s^2(s^2+1))$, $s\in\QQ$;
\item $(u,v) = (\pm 2,1)$.
\end{enumerate}
Later, it was shown \cite{badziahin_2019} that in the first two cases the
value $\limsup_{k\to\infty} d_{k+1}/d_k$ is equal to two while in the third
case it is $\frac75$. It is shown in the same paper that for functions
$g_{u,v}(z)$ the condition $\limsup_{k\to\infty} d_{k+1}/d_k=1$ is equivalent
to $\forall k\in\NN$, $d_k = k$.

It is conjectured \cite[Conjecture A]{badziahin_2017} that $d_k = k$ for all
$k\in\ZZ$ for the other pairs $(u,v)\in\ZZ^2$. This conjecture is
verified\cite{badziahin_2017} for large sets of pairs $(u,v)$. In particular
it is true if $u=0$ or $v=0$ and also for the region
$$
\big\{(u,v)\in\QQ^2\;:\; u^2\ge 6,\; v\ge \max\{3u^2-1,
2u^2+8\}\big\}.
$$
Some local conditions on $u$ and $v$ modulo 3, ensuring $d_k=k$ are
also provided in~\cite{bugeaud_han_wen_yao_2015}. The purpose of
this note is to cover as many other pairs $(u,v)\in\ZZ^2$ as
possible and hence make a contribution to the conjecture above.

The main result of this paper provides a number of local conditions
on $(u,v)$ modulo any prime $p\ge 3$ which ensure that $d_k=k$ for
all $k$. In particular, for $p=3$ they coincide with those
from~\cite{bugeaud_han_wen_yao_2015}.

\begin{theorem}\label{th1}
Let $p\geq 3$ be prime and $(u,v)\in \ZZ^2$ satisfy one of the properties
\begin{gather}
    u^2 \equiv 3,\quad v \equiv 1 \pmod p; \label{case1}\\
    u^2 \equiv -3,\quad v \equiv -1 \pmod p; \label{case2}\\
    u \equiv \pm \varphi,\; v \equiv 0 \pmod p,\quad \mbox{where } \varphi^2 + \varphi + 1 \equiv 0\pmod p; \label{case3}\\
    u \equiv \pm \varphi,\; v \equiv -1 \pmod p,\quad \mbox{where } \varphi^4 + 4\varphi^2 + 1 \equiv 0\pmod p; \label{case4}\\
    u \equiv \pm \varphi,\; v \equiv \delta \pmod p,\quad \mbox{where } \delta^2 - \delta + 1 \equiv 0,\; \varphi^2 \equiv 2\delta \pmod p; \label{case5}\\
    u \equiv 0,\; v \equiv \pm \delta \pmod p,\quad \mbox{where } \delta^2 + \delta + 1 \equiv 0\pmod
    p; \label{case6}\\
    u=\pm 2\delta^2, v = \delta\;(\mathrm{mod}\; p), \; \mbox{where } \delta^2+\delta+1 = 0, \; p\neq 3. \label{case7}
\end{gather}
Then all the partial quotients of the continued fraction for the Mahler
function $g_{u,v}(z)$ are linear.
\end{theorem}

Theorem~\ref{th1} together with~{\bf (B)} imply the following

\begin{corollary}
Let $p\ge 3$ be prime and $(u,v)\in \ZZ^2$ satisfy one of the
properties~2 -- 8 of Theorem~\ref{th1}. Then for any integer $b$
such that $|b|\ge 2$ and $g_{u,v}(b)\neq 0$ one has
$$
\mu(g_{u,v}(b)) = 2.
$$
\end{corollary}


\section{Continued Fractions of Laurent Series}\label{sec_hankel}

Let $\FF$ be a field. Consider the set $\FF[[z^{-1}]]$ of Laurent
series together with the valuation: $||\sum_{k=-d}^\infty
c_kz^{-k}|| = d$, the biggest degree $d$ of $x$ having non-zero
coefficient $c_{-d}$. For example, for polynomials $f(z)$ the
valuation $||f(z)||$ coincides with their degree. It is well known
that in this setting the notion of continued fraction is well
defined. In other words, every $f(z)\in \FF[[z^{-1}]]$ can be
written as
$$
f(z) = [a_0(z), a_1(z),a_2(z),\ldots] = a_0(z) + \frac{1}{a_1(z) +
\frac{1}{a_2(z) + \cdots}} ,
$$
where the $a_i(z)$ are non-zero polynomials of degree at least 1,
$i\in\ZZ_{\ge 0}$. We refer the reader to a nice
survey~\cite{poorten_1998} for more properties of the continued
fractions of Laurent series.

It will be more convenient for us to renormalise this continued
fraction to the form
\begin{equation}
    f(z) = a_0(z) + \frac{\beta_1}{a_1(z) + \frac{\beta_2}{a_2(z) + \frac{\beta_3}{a_3(z) + \cdots}}} =: a_0(z) + \K_{i=1}^\infty \frac{\beta_i}{a_i(z)}
\end{equation}
where $\beta_i\in\FF\setminus\{0\}$ are constants and $a_i(z)\in
\FF[z]$ are monic polynomials for $i\ge 1$.

By analogy with the classical continued fractions over $\RR$, by
$k$'th convergent of $f$ we denote the rational function
$$
\frac{p_k(z)}{q_k(z)}:= a_0(z) + \K_{i=1}^k \frac{\beta_i}{a_i(z)}.
$$
They satisfy the following recurrent relation
\begin{equation}\label{convp}
p_0(z) = a_0(z), \quad p_1(z) = a_0(z)a_1(z) + \beta_1,\quad p_n(z)
= a_n(z)p_{n-1}(z) + \beta_np_{n-2}(z),
\end{equation}
\begin{equation}\label{convq}
q_0(z) = 1, \quad q_1(z) = a_1(z),\quad q_n(z) = a_n(z)q_{n-1}(z) +
\beta_nq_{n-2}(z).
\end{equation}
By $d_k$, we denote the degree (or the valuation) of the denominater $q_k(z)$ of $k$'th convergent of $f(z)$.

By analogy with the classical case of real numbers, we call a series
$f(z)\in\FF[[z^{-1}]]$ badly approximable if there exists an absolute
constant $M$ such that $\forall k\in\NN$, $||a_k(z)||\le M$.
Formulae~\eqref{convp},~\eqref{convq} suggest that $||a_k(z)|| = d_{k} -
d_{k-1}$ therefore an equivalent condition for badly approximable series is
$d_k - d_{k-1} \le M$.

Coming back to the series $g_{u,v}(z)$, it is known (see~\cite[Proposition
1]{badziahin_2017}) that $g_{u,v}(z)$ is badly approximable if and only if
$d_k=k$ for all positive integer $k$. Now the main tool in the proof of
Theorem~\ref{th1} is the following criterion~\cite[Theorems
2,3]{badziahin_2019} which ensures that condition for $g_{u,v}(z)$.

\begin{theoremB}
Let $\vu=(u,v)\in\QQ^2$. Let the sequences $\alpha_i$ and $\beta_i$
of rational numbers be computed by the recurrent formulae
\begin{equation}\label{init_d3}
\begin{array}{lll}
\alpha_1 = -u,&\displaystyle \alpha_2 =
\frac{u(2v-1-u^2)}{v-u^2},&\displaystyle\alpha_3 =
\frac{-u(v-1)}{v-u^2};\\[2ex]
\beta_1 = 1,& \beta_2 = u^2-v,&\displaystyle\beta_3 =
\frac{u^2+u^4+v^3 - 3u^2v}{(v-u^2)^2}.
\end{array}
\end{equation}
and
\begin{equation}\label{recur_d3}
\begin{array}{l}
\displaystyle \alpha_{3k+4} = -u,\quad \beta_{3k+4} =
\frac{\beta_{k+2}}{\beta_{3k+3}\beta_{3k+2}};\\[1ex]
\displaystyle \beta_{3k+5} = u^2 - v - \beta_{3k+4},\quad
\alpha_{3k+5} = u -
\frac{\alpha_{k+2}+uv-\alpha_{3k+2}\beta_{3k+4}}{\beta_{3k+5}}\\[1ex]
\alpha_{3k+6} = u-\alpha_{3k+5},\quad \beta_{3k+6} = v -
\alpha_{3k+5}\alpha_{3k+6}.
\end{array}
\end{equation}
for any $k\in\ZZ_{\ge 0}$. If all algebraic operations in these
formulae are valid and $\beta_i\neq 0$ for all $i\in\ZZ$ then
$$
g_\vu(z) = a_0(z) + \bigk_{i=1}^\infty \frac{\beta_i}{a_i(z)},\qquad
a_i(z) = z+\alpha_i,
$$
that is, all partial quotients of $q_\vu(z)$ are linear.
\end{theoremB}

%
%
%

\section{Proof of Theorem~\ref{th1}}

The main idea of the proof is that if a pair $(u,v)$ satisfies one
of the conditions~2 -- 8 of Theorem~\ref{th1} then the sequences
$\alpha_i$ and $\beta_i$ from Theorem~B2 modulo $p$ follow a nice
pattern and moreover $\beta_i$ never congruent to zero modulo $p$.
That immediately implies that for all $i\in\NN$, $\beta_i\neq 0$ and
Theorem~B2 implies the required result.

While in each of the cases~\eqref{case1} --~\eqref{case7} the
pattern for sequences $\alpha_i$ and $\beta_i$ is different, the
proofs are extremely similar. We will provide a detailed proof in
the case~\eqref{case1} and only outline the proofs in the remaining
cases~\eqref{case2} --~\eqref{case7}.

\begin{lemma}
If $u^2 \equiv 3 \pmod p$ and  $v \equiv 1 \pmod p$ for odd prime $p$, then
the sequences $(\alpha_i)_{i\in\mathbb{N}}$ and $(\beta_i)_{i\in\mathbb{N}}$
are given by the formula for all $k\ge 0$:
\begin{gather}
    \alpha_{3k+1} \equiv -u,\; \alpha_{3k+2} + \alpha_{3k+3} \equiv u \pmod p;\\
    \alpha_{9k+2} \equiv u,\; \alpha_{9k+5} \equiv \alpha_{3k+3},\; \alpha_{9k+8} \equiv 0 \pmod p;\\
    \beta_1 \equiv 1,\; \beta_2 \equiv 2,\; \beta_{k+3} \equiv 1 \pmod
    p.
\end{gather}

\end{lemma}

\begin{proof}
To shorten the notation, in this proof we omit the $\pmod p$ as it
is implied in every congruence. We use the formulae~\eqref{init_d3}
and~\eqref{recur_d3} to compute the first 9 values of $\alpha_i$ and
$\beta_i$:
$$\alpha_1 \equiv -u,\; \alpha_2 \equiv u,\; \alpha_3 \equiv 0,\; \alpha_4 \equiv -u,\; \alpha_5 \equiv 0,$$
$$\alpha_6 \equiv u,\; \alpha_7 \equiv -u,\; \alpha_8 \equiv 0,\
\alpha_9 \equiv u,$$
$$\beta_1 \equiv 1,\; \beta_2 \equiv 2,\; \beta_i \equiv 1 \text{ for } 3\leq i\leq 9.$$
Now we prove by induction that for $k \geq 1$:
$$\alpha_{9k+1} \equiv \alpha_{9k+4} \equiv \alpha_{9k+7} \equiv-u,$$
$$\alpha_{9k+2} \equiv u, \alpha_{9k+3} \equiv 0, \alpha_{9k+5} \equiv \alpha_{3k+3}, \alpha_{9k+6} \equiv \alpha_{3k+2}, \alpha_{9k+8} \equiv 0, \alpha_{9k+9} \equiv u$$
$$\beta_{i} \equiv 1, \quad 9k+1 \leq i \leq 9k+9$$
Which will give the formula we desire. Note that this is the same as
equations (18) - (20), as we have just given explicit formulas for
the terms defined by $\alpha_{3k+3} = u - \alpha_{3k+2}$.

Suppose that the formulas hold for all $0 \leq k \leq n$. Also note
that this implies that up to these values, all every pair
$\alpha_{3m+2},\alpha_{3m+3}$ is either $(0,u)$ or $(u,0)$ modulo
$p$. Now we prove them for $k = n+1$.

First, it is obvious that $\alpha_{9(n+1)+1} \equiv
\alpha_{9(n+1)+4} \equiv \alpha_{9(n+1)+7} \equiv-u$ as they are all
of the form $\alpha_{3k+1}$.

Second, by~\eqref{recur_d3} we have $\beta_{9(n+1)+1} \equiv
\frac{\beta_{3(n+1)+1}}{\beta_{9(n+1)}\beta_{9(n+1)-1}} \equiv
\frac{1}{1\cdot 1} \equiv 1$ by the induction hypothesis. Then
$\beta_{9(n+1) + 2} \equiv u^2 - 1 - \beta_{9(n+1)+1} \equiv 3 - 1 -
1 \equiv 1$.

Third, we compute:
\begin{align*}
    \alpha_{9(n+1)+2} &\equiv u - \frac{\alpha_{3(n+1)+1} + u - \alpha_{9(n+1)-1}\beta_{9(n+1)+1}}{\beta_{9(n+1)+2}} \\
    &\equiv u - (\alpha_{3n+4} + u - \alpha_{9n+8}) \\
    &\equiv u - (-u + u - 0) \equiv u
\end{align*}
This then implies $\alpha_{9(n+1)+3} \equiv u - \alpha_{9(n+1)+2} \equiv 0$. Thus, $\beta_{9(n+1)+3} \equiv 1 - \alpha_{9(n+1)+2}\alpha_{9(n+1)+3} \equiv 1$.

Fourth, we continue in the same way to compute $\beta_{9(n+1)+4}
\equiv \frac{\beta_{3(n+1)+2}}{\beta_{9(n+1)+3}\beta_{9(n+1)+2}}
\equiv \frac{1}{1\cdot 1} \equiv 1$ and $\beta_{9(n+1) + 5} \equiv
u^2 - 1 - \beta_{9(n+1)+4} \equiv 1$.

Now this implies:
\begin{align*}
    \alpha_{9(n+1)+5} &\equiv u - \frac{\alpha_{3(n+1)+2} + u - \alpha_{9(n+1)+2}\beta_{9(n+1)+4}}{\beta_{9(n+1)+5}} \\
    &\equiv u - (\alpha_{3(n+1)+2} + u - \alpha_{9(n+1)+2}) \\
    &\equiv u - (\alpha_{3(n+1)+2} + u - u) \\
    &\equiv u - \alpha_{3(n+1)+2} \equiv \alpha_{3(n+1)+3}.
\end{align*}
This then implies $\alpha_{9(n+1)+6} \equiv u - \alpha_{3(n+1)+3}
\equiv \alpha_{3(n+1)+2}$ and $\beta_{9(n+1)+6} \equiv 1 -
\alpha_{9(n+1)+2}\alpha_{9(n+1)+3} \equiv 1 -
\alpha_{3(n+1)+2}\alpha_{3(n+1)+3} \equiv 1$, as by the induction
hypothesis one of these are $0$ and the other is $u$.

Finally we have $\beta_{9(n+1)+7} \equiv
\frac{\beta_{3(n+1)+3}}{\beta_{9(n+1)+6}\beta_{9(n+1)+5}} \equiv
\frac{1}{1\cdot 1} \equiv 1$ and $\beta_{9(n+1) + 8} \equiv u^2 - 1
- \beta_{9(n+1)+7} \equiv 1$. This implies:
\begin{align*}
    \alpha_{9(n+1)+8} &\equiv u - \frac{\alpha_{3(n+1)+3} + u - \alpha_{9(n+1)+5}\beta_{9(n+1)+7}}{\beta_{9(n+1)+8}} \\
    &\equiv u - (\alpha_{3(n+1)+3} + u - \alpha_{9(n+1)+5}) \\
    &\equiv u - (\alpha_{3(n+1)+3} + u - \alpha_{3(n+1)+3}) \equiv 0.
\end{align*}
This then implies $\alpha_{9(n+1)+9} \equiv u - \alpha_{9(n+1)+8}
\equiv u$ and $\beta_{9(n+1)+9} \equiv 1 -
\alpha_{9(n+1)+8}\alpha_{9(n+1)+9} \equiv 1$.

Thus the formula also holds for $k = n+1$, the proof by inductions
completes.
\end{proof}

For the other cases we provide similar lemmata which can be proven in the
same way by induction. We leave their proof to the interested reader.

\begin{lemma}
If $u^2 \equiv -3 \pmod p$ and $v \equiv -1 \pmod p$ for odd prime $p$, then
the sequences $(\alpha_i)_{i\in\mathbb{N}}$ and $(\beta_i)_{i\in\mathbb{N}}$
are given by the formula for all $k \geq 0$:
\begin{gather*}
    \alpha_{3k+1} \equiv -u,\; \alpha_{3k+2} + \alpha_{3k+3} \equiv u \pmod p;\\
    \alpha_{9k+2} \equiv 0,\; \alpha_{9k+5} \equiv \alpha_{3k+2},\; \alpha_{9k+8} \equiv u \pmod p;\\
    \beta_1 \equiv 1,\; \beta_2 \equiv -2,\; \beta_{k+3} \equiv -1 \pmod p.
\end{gather*}

\end{lemma}

\begin{lemma}\label{lem3}
If $u \equiv \varphi \pmod p$ and $v \equiv 0 \pmod p$ where $\varphi\in\ZZ$
satisfies $\varphi^2 + \varphi + 1 \equiv 0 \pmod p$, then the sequences
$(\alpha_i)_{i\in\mathbb{N}}$ and $(\beta_i)_{i\in\mathbb{N}}$ are given by
the formula for all $k\ge 0$:
\begin{gather*}
    \alpha_{3k+1} \equiv -\varphi,\; \alpha_{3k+2} + \alpha_{3k+3} \equiv \varphi \pmod p;\\
    \alpha_{9k+2} \equiv -1,\; \alpha_{9k+5} \equiv \alpha_{3k+2},\; \alpha_{9k+8} \equiv -\varphi^2 \pmod p;\\
    \beta_1 \equiv 1,\; \beta_2 \equiv \varphi^2,\; \beta_{3k+3} \equiv -\varphi^2,\; \beta_{3k+4} + \beta_{3k+5} \equiv \varphi^2 \pmod p;\\
    \beta_{9k+1} \equiv \beta_{3k+1},\; \beta_{9k+4}\equiv-\varphi,\; \beta_{9k+7} \equiv -1 \pmod p.
\end{gather*}
\end{lemma}

One can easily derive from Lemma~\ref{lem3} that for $i\ge 3$ the value of
$\beta_i$ is congruent to either $-1, -\varphi$ or $-\varphi^2$ modulo $p$,
hence it never equals zero.

\begin{lemma}\label{lem4}
If $u \equiv \varphi \pmod p$ and $v \equiv -1\pmod p$ where $\varphi\in\ZZ$
satisfies $\varphi^4 + 4\varphi^2 + 1 \equiv 0 \pmod p$ and $p$ is an odd
prime, then the sequences $(\alpha_i)_{i\in\mathbb{N}}$ and
$(\beta_i)_{i\in\mathbb{N}}$ are given by the formula for $k\ge 0$:
\begin{gather*}
    \alpha_{3k+1} \equiv -\varphi,\; \alpha_{3k+2} + \alpha_{3k+3} \equiv \varphi \pmod p;\\
    \alpha_{9k+2} \equiv -\varphi^{-1},\; \alpha_{9k+5} \equiv \alpha_{3k+3},\; \alpha_{9k+8} \equiv \varphi+\varphi^{-1} \pmod p;\\
    \beta_1 \equiv 1,\; \beta_2 \equiv \varphi^2 + 1,\; \beta_{3k+3} \equiv \varphi^{-2},\; \beta_{3k+4} + \beta_{3k+5} \equiv \varphi^2 + 1 \pmod p;\\
    \beta_{9k+1} \equiv \beta_{3k+1},\; \beta_{9k+4}\equiv\varphi^2,\; \beta_{9k+7} \equiv 1 \pmod p.
\end{gather*}
\end{lemma}

One can check that for odd prime $p$ the values $\varphi^2$ and $\varphi^2+1$
are not congruent to zero modulo $p$. One can derive from Lemma~\ref{lem4}
that for $i\ge 3$ the value of $\beta_i$ is congruent to either
$\varphi^{-2}, \varphi^2$ or 1 modulo $p$. Hence it never equals zero.

\begin{lemma}\label{lem5}
If $u \equiv \varphi \pmod p$ and $v \equiv \delta \pmod p$ where $\delta,
\varphi\in \ZZ$ satisfy $\delta^2 - \delta + 1 \equiv 0 \pmod p$, $\varphi^2
\equiv 2\delta \pmod p$ and $p$ is an odd prime, then the sequences
$(\alpha_i)_{i\in\mathbb{N}}$ and $(\beta_i)_{i\in\mathbb{N}}$ are given by
the formula for $k\ge 0$:
\begin{gather*}
    \alpha_{3k+1} \equiv -\varphi,\; \alpha_{3k+2} + \alpha_{3k+3} \equiv \varphi \pmod p;\\
    \alpha_{9k+2} \equiv \frac{\varphi}{\delta},\; \alpha_{9k+5} \equiv \alpha_{3k+3},\; \alpha_{9k+8} \equiv \varphi\delta \pmod p;\\
    \beta_1 \equiv 1,\; \beta_2 \equiv \delta,\; \beta_{3k+3} \equiv -\delta,\; \beta_{3k+4} + \beta_{3k+5} \equiv \delta \pmod p\\
    \beta_{9k+1} \equiv \beta_{3k+1},\; \beta_{9k+4}\equiv-\frac{1}{\delta},\; \beta_{9k+7} \equiv 1 \pmod p.
\end{gather*}
\end{lemma}

Lemma~\ref{lem5} implies that for all $i\ge 3$ the value of $\beta_i$ is
congruent to either $-\delta, -\delta^{-1}$ or 1 modulo $p$, hence it never
equals zero.

\begin{lemma}\label{lem6}
If $u \equiv 0 \pmod p$ and $,v \equiv \delta \pmod p$ where $\delta\in \ZZ$
satisfies $\delta^2 + \delta + 1 \equiv 0 \pmod p$, then the sequences
$(\alpha_i)_{i\in\mathbb{N}}$ and $(\beta_i)_{i\in\mathbb{N}}$ are given by
the formula for $k\ge 0$:
\begin{gather*}
    \alpha_{k} \equiv 0 \pmod p;\\
    \beta_1 \equiv 1,\; \beta_2 \equiv -\delta,\; \beta_{3k+3} \equiv \delta,\; \beta_{3k+4} + \beta_{3k+5} \equiv -\delta \pmod p;\\
    \beta_{9k+1} \equiv \beta_{3k+1},\; \beta_{9k+4}\equiv \delta^{-1},\; \beta_{9k+7} \equiv 1 \pmod p.
\end{gather*}

\end{lemma}

Under conditions of lemma~\ref{lem6}, for all $i\ge 3$ the value of $\beta_i$
is congruent to either $\delta, \delta^{-1}$ or 1 modulo $p$. Hence it never equals zero.

\begin{lemma}\label{lem7}
If $u \equiv \pm 2\delta^2 \pmod p$ and $v \equiv \delta \pmod p$ where
$\delta\in \ZZ$ satisfy $\delta^2 + \delta + 1 \equiv 0 \pmod p$ and $p > 3$
is an odd prime, then the sequences $(\alpha_i)_{i\in\mathbb{N}}$ and
$(\beta_i)_{i\in\mathbb{N}}$ are given by the formula for $k\ge 0$:
\begin{gather}
    \alpha_{3k+1} \equiv -u,\; \alpha_{3k+2} + \alpha_{3k+3} \equiv u \pmod p;\label{lem7_eq1}\\
    \alpha_{9k+2} \equiv -\frac{2\delta + 4}{3},\; \alpha_{9k+5} \equiv \frac{u + \alpha_{3k+2}}{3},\; \alpha_{9k+8} \equiv -\frac{4\delta + 2}{3} \pmod p;\\
    \beta_1 \equiv 1,\; \beta_2 \equiv 3\delta,\; \beta_{3k+4} + \beta_{3k+5} \equiv 3\delta \pmod p\\
    \beta_{9k+1} \equiv \beta_{3k+1},\; \beta_{9k+4}\equiv -\frac{3}{\delta},\; \beta_{9k+7} \equiv -3 \pmod p.\\
    \beta_{9k+3} \equiv -\frac{\delta}{3},\; \beta_{9k+6}\equiv \frac{\beta_{3k+3}}{9},\; \beta_{9k+9} \equiv -\frac{\delta}{3} \pmod p\label{lem7_eq5}.
\end{gather}
\end{lemma}

Lemma~\ref{lem7} implies that for all $i\ge 3$ the value of $\beta_i$ is
congruent to either $-\frac{\delta}{3}, -3\delta^{-1}, -3$ or
$\frac{\beta_{i/3+1}}{9}$ modulo $p$, the latter of which is inductively
never zero. Hence none of $\beta_i$ equals zero.

Since the proof of Lemma~\ref{lem7} involves the most tedious computations,
compared to other lemmata, we also outline its proof here.

\begin{proof}
Again we omit the $\pmod p$ in each congruence for this proof. We'll also
just prove it for the $u \equiv 2\delta^2$ case as the proof is essentially
the same. We use the formulae~\eqref{init_d3} and~\eqref{recur_d3} to compute
the first 9 values of $\alpha_i$ and $\beta_i$:
$$
\alpha_1 \equiv -u,\; \alpha_2 \equiv -\frac{2\delta + 4}{3},\; \alpha_3
\equiv -\frac{4\delta + 2}{3},\; \beta_1\equiv 1,\; \beta_2\equiv 3\delta,\; \beta_3\equiv -\frac{\delta}{3}
$$$$
\alpha_4 \equiv -u,\; \beta_4\equiv -\frac{3}{\delta},\; \beta_5\equiv -3,\;
\alpha_5 \equiv -\frac{8\delta + 10}{9},\; \alpha_6 \equiv -\frac{10\delta +
8}{9},\; \beta_6\equiv -\frac{\delta}{27},
$$
$$ \alpha_7 \equiv -u,\; \beta_7\equiv -3,\; \beta_8\equiv -\frac{3}{\delta},\; \alpha_8 \equiv -\frac{4\delta + 2}{3},\
\alpha_9 \equiv -\frac{2\delta + 4}{3},\; \beta_9\equiv -\frac{\delta}{3}.$$
This all clearly satisfy the equations~\eqref{lem7_eq1} --~\eqref{lem7_eq5},
except for $\alpha_5$, so let us check this.
$$
\frac{u+\alpha_2}{3}\equiv \frac{2\delta^2 - \frac{2\delta+4}{3}}{3}\equiv -\frac{8\delta+10}{9}\equiv \alpha_5.
$$
%
The base case $k=0$ has been proved.

Now we assume that the equations~\eqref{lem7_eq1} --~\eqref{lem7_eq5} are
satisfied for $\alpha_i,\beta_i$, $1\le i\le 9k$ and  verify them for
$9k+1\le i\le 9k+9$.


First, it is obvious that $\alpha_{9k+1} \equiv \alpha_{9k+4} \equiv
\alpha_{9k+7} \equiv-u$ as they are all of the form $\alpha_{3k+4}$.

Second, by~\eqref{recur_d3} we have
$$\beta_{9k+1} = \frac{\beta_{3k+1}}{\beta_{9k}\beta_{9k-1}} = \frac{\beta_{3k+1}}{-\frac{\delta}{3}\cdot (3\delta - \beta_{9k-2})} = \beta_{3k+1}$$
by the induction hypothesis. Then the equation~\eqref{recur_d3} implies that
$$
\beta_{9k+1}+\beta_{9k+2}\equiv \beta_{9k+4}+\beta_{9k+5}\equiv\beta_{9k+7}+\beta_{9k+8}\equiv u^2-v\equiv 3\delta.
$$
In particular, this together with $\beta_{9k+1}\equiv \beta_{3k+1}$ implies
that $\beta_{9k+2}\equiv \beta_{3k+2}$.

Third, we compute:
\begin{align*}
    \alpha_{9k+2} &\equiv u - \frac{\alpha_{3k+1} + uv - \alpha_{9k - 1} \beta_{9k + 1}}{\beta_{9k + 2}}\\
    &\equiv 2\delta^2 - \frac{-2\delta^2 + 2 + \frac{4\delta+2}{3} \beta_{9k + 1}}{3\delta - \beta_{9k + 1}} \\
    &\equiv \frac{-(2\delta+4)(3\delta - \beta_{9k+1})}{3(3\delta - \beta_{9(n+1) + 1})}\\
    &\equiv -\frac{2\delta+4}{3}.
\end{align*}
Thus we have $\alpha_{9k+3} \equiv u - \alpha_{9k+2} \equiv -\frac{4\delta +
2}{3}$. Finally, we use the last equation in~\eqref{recur_d3} to compute
$\beta_{9k+3}$:
$$
    \beta_{9k+3} \equiv \delta - \alpha_{9k+2}\alpha_{9k+3} \equiv \delta - \frac{(2\delta+4)(4\delta+2)}{9} \equiv \frac{\delta}{3}.
$$

Fourth, we similarly continue to compute $\beta_{9k+4} \equiv
\frac{\beta_{3k+2}}{\beta_{9k+3}\beta_{9k+2}} \equiv
\frac{\beta_{3k+2}}{-\frac{\delta}{3}\cdot \beta_{3k+2}} \equiv
-\frac{3}{\delta}$. This then implies that $\beta_{9k + 5} \equiv 3\delta +
\frac{3}{\delta} \equiv -3$.

Now we compute:
\begin{align*}
    \alpha_{9k+5} &\equiv u - \frac{\alpha_{3k+2} + u \delta - \alpha_{9k + 2} \beta_{9k + 4}}{\beta_{9k + 5}}\\
    &\equiv 2\delta^2 - \frac{\alpha_{3k+2} + 2 - \frac{2\delta+4}{3} \cdot \frac{3}{\delta}}{-3}\\
    &\equiv \frac{2\delta^2 + \alpha_{3(n+1)+2}}{3} \equiv \frac{u + \alpha_{3(n+1)+2}}{3} \\
\end{align*}
This then implies
$$
    \alpha_{9k+6} \equiv u -\frac{u + \alpha_{3(n+1)+2}}{3} \stackrel{\eqref{lem7_eq1}}\equiv \frac{u + \alpha_{3k+3}}{3}
$$
and then we compute
\begin{align*}
    \beta_{9k+6} &\equiv \delta - \alpha_{9k+5}\alpha_{9k+6}\\
    &\equiv \delta - \frac{2\delta^2 + \alpha_{3k+2}}{3}\cdot \frac{2\delta^2 + \alpha_{3k+3}}{3} \\
    &\equiv \frac{5\delta-2\delta^2(\alpha_{3k+3}+\alpha_{3k+2})-\alpha_{3k+3}\alpha_{3k+2}}{9} \\
    &\stackrel{\eqref{lem7_eq1}}\equiv \frac{\delta-\alpha_{3k+3}\alpha_{3k+2}}{9}\; \stackrel{\eqref{recur_d3}}\equiv\;
    \frac{\beta_{3k+3}}{9}.
\end{align*}

We finish the proof by computing the last triple of $\alpha$'s and
$\beta$'s. We verify that $\beta_{9k+7} =
\frac{\beta_{3k+3}}{\beta_{9k+6}\beta_{9k+5}} \equiv -3$ and
$\beta_{9k + 8} \equiv 3\delta - \beta_{9k+7} \equiv
-\frac{3}{\delta} $. Then we use already known values of
$\alpha_{9k+5}, \beta_{9k+7}, \beta_{9k+8}$ to compute:
\begin{align*}
    \alpha_{9k+8} &\equiv u - \frac{\alpha_{3k+3} + u \delta - \alpha_{9k +5} \beta_{9k + 7}}{\beta_{9k + 8}}\\
    &\equiv 2\delta^2 - \frac{\alpha_{3k+3} + 2 + 2\delta^2 + \alpha_{3k+2} }{-\frac{3}{\delta}}\\
    &\stackrel{\eqref{lem7_eq1}}\equiv -\frac{4\delta + 2}{3}.
\end{align*}
This then implies $\alpha_{9k+9} \equiv u - \alpha_{9k+8} \equiv
-\frac{2\delta+4}{3}$ and
$$
\beta_{9k+9} \equiv \delta - \alpha_{9k+8}\alpha_{9k+9} \equiv
-\frac{\delta}{3}.
$$

This finishes the inductional step, thus the proof by inductions
completes.

\end{proof}

All the cases~\eqref{case1} --~\eqref{case7} are now covered and
Theorem~\ref{th1} concludes. \endproof

\section{Further remarks}

In view of Theorem~\ref{th1} one can ask a natural question:
are~\eqref{case1} --~\eqref{case7} the only local conditions on $u,v$ which
guarantee that all the partial quotients of $g_{u,v}(z)$ are linear? In
attempt to answer this question, we conduct a computer search of all primes
$p$ between 3 and 1000 and all pairs $(u,v)\in\FF_p^2$. The search
reveals that every pair that did not seem to ever produce a value of 0 is of the conditions~\eqref{case1} --~\eqref{case7}.


These findings, while heuristic, suggest that Theorem~\ref{th1} covers all local conditions which guarantee that the series $g_{u,v}$ is badly
approximable.

Also, a quick search reveals that around $82\%$ integer pairs $(u,v) \in
[-1000,1000]^2$ satisfy at least one of the conditions~\eqref{case1}
--~\eqref{case7}. This indicates that the majority of pairs are covered by
Theorem~\ref{th1}. However there are still plenty of pairs for which the
conjecture is still to be verified. One of the smallest such pair is $(u,v) =
(2,-2)$.

\end{document}